\documentclass{amsart}
\usepackage{amsthm,hyperref,amssymb,amsrefs,a4wide}
\usepackage[utf8]{inputenc}
\usepackage{fixme}
\newtheorem{theorem}{Theorem}[section]
\newtheorem{lemma}[theorem]{Lemma}
\newtheorem{corollary}[theorem]{Corollary} 
\newtheorem{definition}[theorem]{Definition} 
\newtheorem{remark}[theorem]{Remark}
\newtheorem{example}[theorem]{Example}
\newtheorem{proposition}[theorem]{Proposition}

\newcommand{\cG}{\mathcal{G}}
\newcommand{\id}{\textrm{id}}
\title{Panov's theorem for weak Hopf algebras}
\subjclass[2000]{}
\author{Christian Lomp}
\address{Department of Mathematics of the Faculty of Science and  CMUP, University of Porto, Rua Campo Alegre, 687, 4169-007 Porto, Portugal}
\email{clomp@fc.up.pt}
\author{Alveri Sant’Ana}
\author{Ricardo Leite dos Santos}
\address{Instituto de Matemática e Estatística, Universidade Federal do Rio Grande do Sul, Av. Bento Gonçalves, 9500, 91.509-900, Porto Alegre, Brazil}
\email{alveri@mat.ufrgs.br}
\email{rilesantos@gmail.com}
\thanks{The first author is a member of CMUP (UID/MAT/00144/2013), which is funded by FCT with national (MEC) and European structural funds (FEDER), under the partnership agreement PT2020.}
\thanks{This work emerged out of the project ``Internacionalização da Pós-Graduação no Rio Grande do Sul", financed by FAPERGS, and is part of the last author's doctoral thesis. The last author was partially supported by Capes, Brazil. Some results of this work were obtained during the last author's visit to the University of Porto. He would like to thank the Mathematics Department of the Faculty of Science of the University of Porto for its hospitality. 
}
\begin{document}
\maketitle

\begin{abstract}
Panov proved necessary and sufficient conditions to extend the Hopf algebra structure of an algebra $R$ to an Ore extension $R[x;\sigma,\delta]$ with $x$ being a skew-primitive element. In this paper we extend Panov's result to Ore extensions over weak Hopf algebras. As an application we study Ore extensions of connected groupoid algebras.
\end{abstract}

\section{Introduction}

In \cite{Panov}, Panov found necessary and sufficient conditions for an Ore extension  $H=R[x;\sigma,\delta]$ over a Hopf algebra $R$, to have a structure of a Hopf algebra which extends the Hopf algebra structure of $R$ and such that $x$ is a skew-primitive element, i.e. $\Delta(x)=g\otimes x + x \otimes h$, for some $g,h\in R$. 
Later in \cite{Brown}, Brown et al. extended Panov's result by allowing $x$ to be more general.
In this paper we extend Panov's characterisation to Ore extensions of weak Hopf algebras, where by weak Hopf algebras we mean a generalisation of Hopf algebras in the sense of Böhm et al. \cite{Boehm} see also \cite{BrzezinskiWisbauer, NikshychVainerman}. A weak Hopf algebra $R$ is a (unital, associative) algebra and a (counital, coassociative) coalgebra, with counit $\epsilon$ and comultiplication $\Delta$, satisfying certain conditions. The difficulty for Panov's characterisation to carry over to weak Hopf algebras lies in the fact that the counit $\epsilon$ of a weak Hopf algebra $R$ is not multiplicative and that the comultiplication $\Delta$ does not map $1$ to $1\otimes 1$. Indeed it is well-known that a weak Hopf algebra is a Hopf algebra if and only if $\Delta(1)=1\otimes 1$ if and only if $\epsilon$ is an algebra map (see \cite[page 391]{Boehm}). Hence if $R$ is a weak Hopf algebra over a field $k$ and $x$ is a primitive element (in the usual sense), i.e. $\Delta(x)=1\otimes x + x \otimes 1$, then the coassociativity implies 
$$\Delta(1)\otimes x + 1\otimes x \otimes 1 + x\otimes 1\otimes 1 =  1\otimes 1\otimes x +  1\otimes x \otimes 1 + x\otimes \Delta(x)$$
If $x$ and $1$ are linearly independent, then comparing the tensorands on both sides implies  $\Delta(1)=1\otimes 1$, which shows that $R$ is a Hopf algebra. Note that the counity implies $\epsilon(x)=0$. Hence $x$ and $1$ are only linearly dependent if $x=0$. Thus we will not be able of extending Panov's characterisation with the ordinary definition of a primitive element.

Recall that if $R$ is a Hopf algebra, then $R^*$ is an algebra with the convolution product and $R$ is an $R^*$-bimodule with the left and right action given by $ f\cdot a := a_1f(a_2)$ and $a\cdot f := f(a_1)a_2$ for any $a\in R$ and $f\in R^*$. Let $X(R)$ denote the set of characters of $R$, i.e. the set of $\chi \in R^*$ that are algebra maps. Following  \cite[I.9.25]{BrownGoodearl}, the left (resp. right) winding automorphisms of $R$ are the automorphisms $\tau_\chi^l$ (resp. $\tau_\chi^r$) given by
$\tau_\chi^l(a) = a\cdot \chi = \chi(a_1)a_2$ (resp. $\tau_\chi^r(a)=\chi\cdot a = a_1\chi(a_2)$), for $a\in R$.

In \cite[Theorem 1.3]{Panov}, Panov showed that an Ore extension $H=R[x;\sigma,\delta]$ of a Hopf algebra $R$ carries also a Hopf algebra structure which extends the one of $R$ and such that $x$ is skew-primitive with $\Delta(x)=g\otimes x + x\otimes 1$ for some group-like element $g\in R$ if and only if there exists a character $\chi:R\rightarrow k$ such that 
$\sigma = \tau_\chi^r = Ad_g\tau_\chi^l$ and $\Delta(\delta(a)) = ga_1\otimes \delta(a_2) + \delta(a_1)\otimes a_2$
for all $a\in R$. Here $Ad_g$ is the conjugation by $g$, i.e. $Ad_g(a)=gag^{-1}$ for all $a\in R$.

Note that while any element $g$ in a Hopf algebra $H$ that satisfy $\Delta(g)=g\otimes g$ is a group-like element, and hence invertible, the identity maps $e_x$, for an object $x$ of a groupoid $\cG$, satisfy $\Delta(e_x)=e_x \otimes e_x$, but are not invertible elements in the groupoid algebra $k[\cG]$ unless $\cG$ has only one object and hence $k[\cG]$ is actually a group algebra and therefore a ordinary Hopf algebra. This means that additional care has to be taken in defining group-like and skew-primitive elements for weak Hopf algebras.

\section{Group-like elements in and characters of weak bialgebras}

We start this section recalling the concepts of weak bialgebras and weak Hopf algebras. For more details and properties of these structures we refer \cite{Boehm,NikshychVainerman}. 

\begin{definition}\label{def_weakbialgebra} A weak bialgebra is a $k$-vector space $R$ with the structures of an associative unital algebra and a coassociative counital coalgebra with comultiplication $\Delta$ and counit $\epsilon$ such that:
\begin{itemize}
\item[(i)] The comultiplication $\Delta$ is a multiplicative $k$-linear map such that
$$(\Delta \otimes \emph{id})(\Delta(1)) = (\Delta(1) \otimes 1)(1 \otimes \Delta(1)) = (1 \otimes \Delta(1))(\Delta(1) \otimes 1);$$
\item[(ii)] The counit $\epsilon$ is a $k$-linear map which is weak multiplicative in the sense of
$$\epsilon(fgh) = \epsilon(fg_1)\epsilon(g_2h) = \epsilon(fg_2)\epsilon(g_1h), \ \textrm{ for all } f, g, h \in R.$$
\end{itemize}
\end{definition}
We recall the linear endomorphisms $\epsilon_t, \epsilon_s, \epsilon_t'$ and $\epsilon_s'$ on a weak bialgebra $R$:
$$\epsilon_t(r)= \epsilon(1_1r)1_2, \qquad \epsilon_s(r)= \epsilon(r1_2)1_1, \qquad \epsilon_t'(r)= \epsilon(r1_1)1_2, \qquad \epsilon_s'(r)= \epsilon(1_2r)1_1,$$
for all $r\in R$, where we are using the Sweedler notation without summation symbol $\Delta(r) = r_1\otimes r_2$, for any $r\in R$. Moreover, any weak bialgebra $R$ has the two subalgebras
$$R_s := \mathrm{Im}(\epsilon_s) = \{a\in R \mid\Delta(a)=1_1\otimes a1_2 \}, \qquad R_t := \mathrm{Im}(\epsilon_t) = \{a\in R \mid\Delta(a)=1_1a\otimes 1_2 \},$$
which are separable algebras over the base field $k$ (see \cite[36.6, 36.8]{BrzezinskiWisbauer}).
\begin{definition}
 A weak Hopf algebra is a weak bialgebra $R$ with multiplication $\mu$ and comultiplication $\Delta$, such that there exists a linear map $S: R \rightarrow R$, named antipode, that satisfies $$\epsilon_t=\mu (id\otimes S) \Delta, \qquad \epsilon_s=\mu (S\otimes id) \Delta \qquad \mbox{ and } \qquad S(r_1)r_2S(r_3) = S(r), \mbox{ for all } r\in R.$$
\end{definition}

 Thus, if $R$ is a weak Hopf algebra with antipode $S$, then $\epsilon_t(r)=r_1S(r_2)$ and $\epsilon_s(r)=S(r_1)r_2$ holds. Moreover, using the weak counity property \ref{def_weakbialgebra}(ii) one easily deduces for any $a,b\in R$:
\begin{equation}\label{wcounit}\epsilon(ab)=\epsilon(a\epsilon_t(b))=\epsilon(a\epsilon_s'(b))=\epsilon(\epsilon_t'(a)b)=\epsilon(\epsilon_s(a)b)\end{equation}
Hence if $\epsilon_t(b)=1$, then $\epsilon(ab)=\epsilon(a)$ and if $\epsilon_s(a)=1$ then $\epsilon(ab)=\epsilon(b)$.  Using equation (\ref{wcounit}) and induction one proves:

\begin{lemma}\label{property_et1}
Let $R$ be a weak bialgebra and $g\in R$. Then:
\begin{itemize}
\item[(i)] $\epsilon_t(g)=1$ if and only if $\epsilon(ag^n)=\epsilon(a)$, for any $a\in R$ and $n\geq 0$;
\item[(ii)] $\epsilon_s(g)=1$ if and only if $\epsilon(g^na)=\epsilon(a)$, for any $a\in R $and $n\geq 0$.
\end{itemize}
\end{lemma}

The same is true if we change $\epsilon_t$ by $\epsilon_s^{\prime}$ in $(i)$ and $\epsilon_s$ by
$\epsilon_t^{\prime}$ in $(ii)$ of Lemma above.

\subsection{Weak group-like Elements} As mentioned before, any element $g$ in a Hopf algebra $H$ that satisfy $\Delta(g)=g\otimes g$ is a group-like element, and hence invertible. However this might not be anymore the case for weak Hopf algebras and therefore additional care has to be taken in defining group-like elements for weak Hopf algebras.

\begin{definition}[{cf. \cite[Definiton 4.1]{Nikshych}}]
An element $g$ of a weak bialgebra is called {\it weak group-like} if $\Delta(g)=\Delta(1)(g\otimes g)$ and $\Delta(g)= (g\otimes g)\Delta(1)$ hold. A {\it group-like element} is an invertible weak group-like element.
\end{definition}

If $R$ is a weak bialgebra, then we denote the set of weak group-like elements of $R$ by $G_{w}(R)$.  Thus, $G(R):=G_w(R)\cap \mathcal{U}(R)$, is the set of group-like elements of $R$, where $\mathcal{U}(R)$ denotes the unit group of $R$.
Note that $1$ is always a group-like element in any weak bialgebra. Furthermore it is easy to check that $G_w(R)$ is  a monoid with the multiplication of $R$ and that $G(R)$ is a subgroup of $\mathcal{U}(R)$.

\begin{example}\label{example_matrix}
Let $\cG=(\cG_0,\cG_1)$ be a groupoid with $|\cG_0|$ finite. Let $R=k\cG_1$ be its groupoid algebra, which is a weak Hopf algebra. Then any element $g\in \cG_1$ satisfies  $\Delta(g)=g\otimes g$ and  as $\Delta(g)=\Delta(1)\Delta(g)=\Delta(g)\Delta(1)$, $g$ is a weak group-like element. A particular groupoid algebra is $R=M_n(k)$, the $n\times n$ matrix ring over a field $k$. Here $|\cG_0|=n$ and for any $i,j\in \cG_0$ there exist exactly one morphism $E_{ij}$ from $i$ to $j$. Clearly $E_{ii}$ are the identity maps and it is easy to check that $E_{ij}E_{st}=\delta_{j,s}E_{it}$ satisfies the ordinary multiplication rule of the elementary units of the $n\times n$ matrix ring $M_n(k)$. 
Thus we can identify $R=k\cG_{1}$ with $M_n(k)$ such that 
$$\Delta(E_{ij})=E_{ij}\otimes E_{ij}, \qquad \epsilon(E_{ij})=1, \qquad \mbox{and} \qquad S(E_{ij})=E_{ji}.$$ 
In particular, the monoid $G_w(R)$ of weak group-like elements of $R=M_n(k)$ is given by 
$$G_w(M_n(k)) = \left\{ \sum_{i\in I} E_{i\sigma(i)} \mid I\subseteq \{1,\ldots, n \}, \:\: \sigma:I\rightarrow \{1,\ldots, n\} \mbox{ injective }\right\}.$$
The group-like elements are precisely the permutation matrices, i.e. the elements of the form $g_\sigma= \sum_{i=1}^n E_{i\sigma(i)}$ for some $\sigma \in S_n$. The assignment $\sigma \mapsto g_\sigma$ is  an isomorphism of groups between the symmetric group $S_n$ and $G(M_n(k))$. 

To see the above equality, we take $\gamma=\sum_{i,j=1}^n \lambda_{ij} E_{ij} \in G_w(R)$. Since $\Delta(\gamma)=\Delta(1)(\gamma\otimes \gamma)$ we have that $\sum_{i,j=1}^n \lambda_{ij} E_{ij}\otimes E_{ij} = \sum_{i,j,t=1}^n \lambda_{ij}\lambda_{it} E_{ij}\otimes E_{it}$. Consequently, we obtain $\lambda_{ij}^2=\lambda_{ij}$ and $\lambda_{ij}\lambda_{it}=0$ whenever $t\neq j$, for all $1\leq i,j,t \leq n$,
and it follows that $\lambda_{ij}\in \{0,1\}$. Moreover, if $\lambda_{ij}\neq 0$ then  $\lambda_{it}=0$, for all $t\neq j$. Set
$I=\{i\in\{1,\ldots, n\}: \exists j: \lambda_{ij}=1\}$. Then we can define a map $\sigma : I \rightarrow \{1,\ldots, n\}$ by $\sigma(i) = j$ if and only if $\lambda_{ij} = 1$ and it follows that $\gamma = \sum_{i\in I} E_{i\sigma(i)}$. The reverse inclusion is clear.
\end{example}

Using the counity properties we have the following easy Lemma:

\begin{lemma}\label{basics_grplke} Let $g$ be  a weak group-like element in a weak bialgebra $R$. Then $g = \epsilon_t(g)g = g\epsilon_s(g).$
If $R$ is a weak Hopf algebra then $\epsilon_t(g)=gS(g)$ and $\epsilon_s(g)=S(g)g$ are idempotents.
\end{lemma}

It follows from the last Lemma that for any weak group-like element $g$ that has a right inverse $\epsilon_t(g)=\epsilon_s'(g)=1$ holds, while $\epsilon_s(g)=\epsilon_t'(g)=1$ holds if $g$ has a left inverse. Consequently, if $R$ is a weak Hopf algebra, then $S(g)$ is a right inverse of a weak group-like element $g$ if and only if $\epsilon_t(g)=1$ as well as $S(g)$ is a left inverse of $g$ if and only if $\epsilon_s(g)=1$. The example of the matrix units of $R=M_n(k)$ shows that in general $\epsilon_t(g)$ might be different from $1$ for a weak group-like element $g$.

\subsection{Weak Characters}

A character  $\chi$ of a Hopf algebra $R$ is a (unital) algebra homomorphism $\chi:R\rightarrow k$. In particular, characters are group-like elements of the finite dual $R^\circ$ of a Hopf algebra. One particular example of a character of a Hopf algebra is its counit $\epsilon \in R^*$. In the case of a weak bialgebra, the counit is not an algebra homomorphism as it is neither unital nor multiplicative. 
However the notion of a left (respectively right) winding map $\tau_\chi^l$ (resp $\tau_\chi^r$) as mentioned in the introduction (see also  \cite[I.9.25]{BrownGoodearl}) makes sense for weak bialgebra. 

\begin{definition} A weak right character of a weak bialgebra $R$ is a linear map $\chi: R\rightarrow k$ such that $\tau_\chi^r$ is a unital algebra homomorphisms. We denote the set of all weak right characters by $X_w^r(R)$. 
\end{definition}

\begin{lemma}\label{chi_rs_linear_part2} Let $R$ be a weak bialgebra.
Then $X_w^r(R)$ is a submonoid of the multiplicative monoid $R^*$. Moreover, for any $\chi \in X_w^r(R)$:
	\begin{enumerate}
\item $\tau_\chi^r(a)=a$ for any $a\in R_t$ and $\chi\in X_w^r(R)$.
\item If $\chi$ is invertible in $R^*$ with inverse $\chi'$, then $\tau_\chi^r$ is an automorphism of $R$ and $\chi'\in X_w^r(R).$
\item If $\tau_\chi^r$ is an automorphism of $R$, then $\chi$ has a left inverse $\chi'$ in $R^*$.
	\end{enumerate}
\end{lemma}
\begin{proof}
 Since $R$ is a coalgebra,  $\tau_\epsilon^r = id$, i.e. $\epsilon \in X_w^r(R)$.
	Let $\chi$ and $\chi'$ be any elements of $R^*$. For any $a\in R$ we have
	\begin{equation}\label{composition_char}\tau_{\chi*\chi'}^r (a) = a_1 \chi*\chi' (a_2) = a_1\chi(a_2)\chi'(a_3) = \tau_\chi^r(a_1\chi'(a_2))=\tau_\chi^r\left ( \tau_{\chi'}^r(a)\right).\end{equation}
	Hence if $\chi$ and $\chi'$ are in $X_w^r(R)$, then  $\tau_{\chi*\chi'}^r = \tau_\chi^r\circ \tau_{\chi'}^r$ is a unital algebra map.
$(1)$ Let $a\in R_t$. Then $\Delta(a)=1_1a\otimes 1_2$ (see \cite[36.6]{BrzezinskiWisbauer}).
Hence $\tau_\chi^r(a)=1_1a\chi(1_2)=\tau_\chi^r(1)a=a$, since  $\tau_\chi^r$ is unital. 

$(2)$ This follows from equation (\ref{composition_char}), because if $\chi'$ is a (two-sided) inverse of $\chi$ in $R^*$, then
$$\tau_\chi^r\circ\tau_{\chi'}^r = \tau_{\chi * \chi'}^r = \tau_\epsilon = id = \tau_\epsilon = \tau_{\chi' * \chi}^r  = \tau_{\chi'}^r\circ\tau_\chi^r.$$
Hence, $\tau_{\chi'}^r$ is the inverse function of the algebra map $\tau_\chi^r$ and therefore itself an algebra map, i.e. $\chi'\in X_w^r(R)$.

(3) If $\tau_\chi^r$ is an automorphism with inverse $\sigma$, then $\chi'=\epsilon\sigma$ satisfaz
$$\chi'*\chi(a) = \chi'(a_1)\chi(a_2) = \epsilon\sigma(\tau_\chi^r(a))= \epsilon(a),\qquad \forall a\in R.$$ Hence $\chi'$ is a left inverse of $\chi$ in $R^*$.
\end{proof}

	Analogously to the definition of a weak right character, we can define a weak left character as an element $\chi\in R^*$ such that $\tau_\chi^l$ is a unital algebra homomorphism. The set of such elements shall be denoted by $X_w^l(R)$.  It is clear that left and right characters are the same if $R$ is cocommutative.
With the same proof as in the last Lemma we have

\begin{lemma}\label{chi_rs_linear_part3} Let $R$ be a weak bialgebra.
Then $X_w^l(R)$ is a submonoid of the multiplicative monoid $R^*$. Moreover, for any $\chi \in X_w^r(R)$:
	\begin{enumerate}
\item $\tau_\chi^l(a)=a$ for any $a\in R_s$ and $\chi\in X_w^l(R)$.
\item If $\chi$ is invertible in $R^*$ with inverse $\chi'$, then $\tau_\chi^l$ is an automorphism of $R$ and $\chi'\in X_w^l(R).$
\item If $\tau_\chi^l$ is an automorphism of $R$, then $\chi$ has a right inverse $\chi'$ in $R^*$.
	\end{enumerate}

\end{lemma}
	
	Furthermore, we set $X_w(R) = X_w^l(R)\cap X_w^r(R)$, whose elements are called {\it weak characters} and $X(R)=U(R^*)\cap X_w(R)$, whose elements are called {\it characters}.

\begin{example} \label{characters_matrix} Let $R=M_n(k)$ be the weak Hopf algebra from Example \ref{example_matrix} and take $\chi\in X_w(R)$. As $R$ is cocommutative, $\sigma:=\tau_\chi^l = \tau_\chi^r$ is a unital algebra homomorphism.
For each $i,j$ there are elements $\lambda_{ij}:=\chi(E_{ij})$. In particular, $\sigma(E_{ij})=  \lambda_{ij}E_{ij}$ and hence 
$\sum_i E_{ii} = 1 = \sigma(1) =  \sum_k \lambda_{ii} E_{ii}$, i.e. $\lambda_{ii}=1$ for all $i$. Moreover, 
$$\lambda_{il}\lambda_{lj} E_{ij}=  \sigma(E_{il})\sigma(E_{lj}) = \sigma(E_{il}E_{lj}) = \sigma(E_{ij}) = \lambda_{ij}E_{ij}.$$
Thus $\lambda_{ij} = \lambda_{il}\lambda_{lj}$ for all $i,j,l$. The case $j=i$ shows that $1 = \lambda_{ii} = \lambda_{il}\lambda_{li}$, i.e. $\lambda_{il}=\lambda_{li}^{-1} $, for any $i,j$. Thus  $\lambda_{ij} = \lambda_{il}\lambda_{lj} = \lambda_{li}^{-1}\lambda_{lj}$, which holds for any $l$. Fixing $l=1$ one can show that $\lambda_{12}, \ldots, \lambda_{1n}$ uniquely determine all such elements $\lambda_{ij}$ and hence the algebra map $\sigma$. Conversely, it is clear that  for any  $q_1=1, q_2,\ldots, q_n \in \mathcal{U}(k)=:k^\times$ one can define a weak character $\chi$ by setting $\chi(E_{ij}) := q_i^{-1}q_j$. Note that any such weak character is actually invertible. Hence we have shown that $X_w(R)=X(R)\simeq (k^\times)^{n-1}$ as groups. Furthermore, if $n>1$, then the characters $\chi$ defined like this are not multiplicative, since
$\chi(E_{11}E_{22})=0$, while $\chi(E_{11})\chi(E_{22})=1$.
\end{example}

The following Proposition is important for Panov's theorem.
\begin{proposition}\label{charcterisation_char} Let $R$ be a weak bialgebra and $\sigma:R\rightarrow R$ a unital algebra homomorphism.
\begin{enumerate}
\item  $\sigma = \tau_\chi^r$ for some $\chi \in X_w^r(R)$  if and only if $\Delta\sigma = (\emph{id} \otimes \sigma)\Delta$.
\item  $\sigma = \tau_\chi^l$ for some $\chi \in X_w^l(R)$  if and only if $\Delta\sigma = (\sigma \otimes \emph{id})\Delta$.
\end{enumerate}
In both cases $\chi=\epsilon\sigma$.
\end{proposition}

\begin{proof}
(1) Let $\sigma=\tau_\chi^r$ for some $\chi\in X_w^r(R)$. Then for any $a\in R$:
$$\Delta\sigma(a)=a_1\otimes a_2\chi(a_3) = a_1\otimes \tau_\chi^r(a_2) = (\id\otimes \sigma)\Delta(a).$$
On the other hand if $\sigma$ satisfies $\Delta\sigma=(\id\otimes \sigma)\Delta$, then applying $\id\otimes \epsilon$ and setting $\chi=\epsilon\sigma \in R^*$ yields
for all $a\in R$: $$\sigma(a) = (\id\otimes \epsilon)\Delta(\sigma(a)) = (\id\otimes \chi)\Delta(a) = a_1\chi(a_2)=\tau_\chi^r(a).$$

The proof of (2) is similar to (1).
\end{proof}

\begin{corollary}\label{chi_rs_linear} Let $R$ be a weak Hopf algebra with antipode $S$, then for any $\chi \in X_w(R)$:
\begin{enumerate}
\item[(i)] $S*\tau_\chi^r = \epsilon_s\tau_\chi^r$ and $\tau_\chi^l*S = \epsilon_t\tau_\chi^l$;
\item[(ii)] If $\chi S$ is the inverse of $\chi$ in $R^*$ then $S=\tau_\chi^l S \tau_\chi^r = \tau_\chi^r S \tau_\chi ^l.$
\end{enumerate}
\end{corollary}

\begin{proof}
$(i)$ For any $a\in R$:
$S*\tau_\chi^r(a) = S(a_1)a_2\chi(a_3) = \epsilon_s(a_1\chi(a_2))=\epsilon_s(\tau_\chi^r(a))$. Similarly 
$\tau_\chi^l(a)*S = \chi(a_1)a_2S(a_3) = \epsilon_t(\chi(a_1)a_2)=\epsilon_t(\tau_\chi^l(a))$. 

$(ii)$  Suppose that $\chi S$ is the inverse of $\chi$. We calculate $$\tau_\chi^l(S(\tau_\chi^r(a)))=S(a_1)\chi(S (a_2)) \chi(a_3) = S(a_1) (\chi S*\chi)(a_2) = S(a),$$ for any $a\in R$. Similarly 
$\tau_\chi^r(S(\tau_\chi^l(a)))=\chi(a_1)\chi(S (a_1)) S(a_3) = (\chi *\chi S)(a_1)S(a_2) = S(a)$.

\end{proof}


\subsection{Tensor Products}\label{tensorproducts} The tensor product of two weak bialgebras (resp. weak Hopf algebras) is again a weak bialgebra (resp. Hopf algebra). More precisely if $A$ and $B$ are weak bialgebras with coalgebra structures $(A,\Delta_A, \epsilon_A)$ and $(B,\Delta_B, \epsilon_B)$, then the algebra $A\otimes B$ has the coalgebra structure $\Delta(a\otimes b) = \mathrm{flip} (\Delta_A\otimes \Delta_B)(a\otimes b) = (a_1\otimes b_1) \otimes (a_2\otimes b_2)$ and counit $\epsilon(a\otimes b) =\epsilon_A(a)\epsilon_B(b)$. If $A$ and $B$ have antipodes $S_A$ and $S_B$, respectively,  then $S(a\otimes b)=S_A(a)\otimes S_B(b)$ defines an antipode on $A\otimes B$.
In particular $G_w(A\otimes B) = G_w(A)\otimes G_w(B)$ and $X_w(A\otimes B)=X_w(A)\otimes X_w(B)\subseteq A^*\otimes B^*$, where for $\chi \in X_w(A)$ and $\chi' \in X_w(B)$ one naturally means by $\chi''=\chi \otimes \chi'$ the map that sends $a\otimes b$ to $\chi(a)\chi'(b)$. 

Note that for any Hopf algebra $H$, the matrix ring $M_n(H)=M_n(k)\otimes H$ is a weak Hopf algebra that is a Hopf algebra if and only if $n=1$.

Going back to the example of a groupoid algebra $R=k\cG$,  Example \ref{example_matrix}, we note that any groupoid can be decomposed in its connected component, i.e. the partition of the vertices $\cG_0$ such that between two different parts there does not exist any morphism. Then $R$ decomposes into a direct product of groupoid algebras of its connected components. Furthermore if $\cG$ is connected and $e_1$ is any vertex in $\cG_0$, then $G=\mathrm{Aut}(e_1)$, the set of automorphisms of $e_1$ is a group. Suppose that $\cG_0=\{e_1,\ldots, e_n\}$ is the set of vertices. For any $i>1$ fix a morphism $\alpha_{i}$ from $e_1$ to $e_i$. Then any morphism $\beta$ from $e_i$ to $e_j$ is of the form $\beta = \alpha_i^{-1}g \alpha_j $, where  $g=\alpha_i\beta\alpha_j^{-1}\in G$ (where we write the composition of maps from left to right, as in the case of going along the arrows in a diagram). Hence the map $\beta \mapsto gE_{ij}$ defines an algebra automorphism from the groupoid algebra $k\cG$ to the matrix ring $M_n(kG)$ over the group algebra $kG$. Hence we will identify $R$ with $M_n(kG)$ and as such $R=M_n(k)\otimes kG$ is a tensor product of the Hopf algebra $kG$ and the weak Hopf algebra $M_n(k)$. In particular given any group character $\rho:G\rightarrow k^\times$ and elements $q_1=1, q_2,\ldots, q_n \in k^\times$ (see Example \ref{characters_matrix}) we have a character $\chi$ of $R$ defined by $\chi(gE_{ij}) = q_i^{-1}q_j \rho(g)$.

\section{Skew-primitive elements in and skew-coderivations of weak bialgebras}

\subsection{Skew-primitive elements} A weak Hopf algebra that contains a primitive element in the usual sense must be a Hopf algebra, as mentioned in the introduction. We propose the following definition.

\begin{definition}
Let $H$ be a weak bialgebra and $g,h \in H$ two weak group-like elements.
An element $x\in H$ is called $(g,h)$-primitive if $\Delta(x)=\Delta(1)(g\otimes x + x\otimes h)$ and $\Delta(x)=(g\otimes x + x\otimes h)\Delta(1)$. An element is called skew-primitive if it is  $(g,h)$-primitive for some weak group-like elements $g,h$.
\end{definition}

\begin{remark} Let $R\subseteq H$ be an extension of weak Hopf algebras, $x\in H$ and $g,h \in R$ such that 
$\Delta(1)(g\otimes x + x\otimes h) = \Delta(x) = (g\otimes x + x\otimes h)\Delta(1).$ Suppose $R\cap Rx=0=R\cap xR$. The coassociativity of $\Delta$ applied on the first equality shows that
$$\Delta^2(1)(g\otimes g\otimes x + g\otimes x\otimes h + x\otimes \Delta(h))=(\Delta\otimes \id)\Delta(x)=\Delta^2(1)(\Delta(g)\otimes x + g\otimes x\otimes h + x\otimes h \otimes h).$$
Hence comparing the coefficient of $1\otimes 1 \otimes x$ in both expressions leads to $\Delta^2(1)(g\otimes g \otimes 1)=\Delta^2(1)(\Delta(g)\otimes 1)$. Applying $\id\otimes \epsilon$ yields $\Delta(1)(g\otimes g)=\Delta(g)$. Analogously,  using the second equality, one concludes $\Delta(g)=(g\otimes g)\Delta(1)$ which shows that $g$ needs to be a weak group-like element. The same is true for $h$, i.e. $x$ is $(g,h)$-primitive.
\end{remark}

\begin{lemma}\label{lemma_epsilon_t_x} Let $R$ be a weak bialgebra and $x$ a  $(g,h)$-primitive element.
Then
$$x = \epsilon_t(g)x + \epsilon_t(x)h  = g\epsilon_s(x) + x\epsilon_s(h).$$
\end{lemma}

\begin{proof}
This follows  from the counity condition and the comultiplication of $x$.
\end{proof}

\begin{remark}\label{remark_epsilon_t}
As a consequence from the last Lemma one concludes that $\epsilon_t(x)=0$ if 
 $\epsilon_t(g)=1$ and $h$ is left $R_t$-torsion free, while $\epsilon_s(x)=0$ if $\epsilon_s(h)=1$ and $g$ is right $R_s$-torsion free.
 
 Moreover if $R\subseteq H$ is an extension of weak bialgebras such that $x\in H$, $g,h\in R$ and  $R\cap Rx=R\cap xR=0$, then $x= \epsilon_t(g)x = x\epsilon_s(h)$ and $\epsilon_t(x)h=g\epsilon_s(x)=0.$
Thus if $x$ is right $R_t$-torsion free, $\epsilon_t(g)=1$, which would imply that $g$ has a right inverse in case $R$ is a weak Hopf algebra. Similarly if $x$ is left $R_s$-torsion free, then $\epsilon_s(h)=1$, which would imply that $h$ has a left inverse in case of $R$ being a weak Hopf algebra.
\end{remark}

\begin{example}
Let $H$ be a Hopf algebra and consider the weak Hopf algebra $R=M_n(H)$ for some $n>0$. The coalgebra structure of $R$ is given by 
(cf. Subsection \ref{tensorproducts})  
$$\Delta(hE_{ij})=\Delta(h)(E_{ij}\otimes E_{ij}), \qquad \epsilon(hE_{ij})=\epsilon(h), \qquad S(hE_{ij})=S(h)E_{ji}, \forall h\in H, \forall i,j$$

Let $x\in H$ be a primitive element. Then  $xE_{ij}$ is a $(E_{ij},E_{ij})$-primitive element of $R$ since
$\Delta(xE_{ij}) = E_{ij}\otimes xE_{ij} + xE_{ij}\otimes E_{ij}.$ 
Note that if $n>1$, then neither is $E_{ij}$ invertible, i.e. $\epsilon_t(E_{ij})\neq 1$, nor is $E_{ij}$ a left $R_t$-torsion-free element.  However $ \epsilon(xE_{ij})=0$.

\end{example}

\subsection{Coderivations}

A derivation $\delta$ that satisfies Panov's condition is called a coderivation. We briefly discuss basic properties of such maps. Given a coalgebra $C$, a {\it coderivation} on $C$ is a map $\delta: C\rightarrow C$ such that $\Delta\delta = (\id\otimes \delta + \delta\otimes \id)\Delta$ holds (see \cite[p.44]{Doi}). For any $\chi\in C^*$ the map $\delta:=(\id\otimes \chi - \chi\otimes \id)\Delta$ is a coderivation, since
\begin{eqnarray*}\Delta\delta(a)&=&\Delta(a_1\chi(a_2)-\chi(a_1)a_2)\\
&=&a_1\otimes a_2\chi(a_3) - \chi(a_1)a_2\otimes a_3\\
&=&a_1\otimes a_2\chi(a_3) - a_1 \otimes \chi(a_2)a_3 + a_1\chi(a_2)\otimes a_3 - \chi(a_1)a_2\otimes a_3\\
&=& a_1\otimes \delta(a_2) + \delta(a_1)\otimes a_2 = (\id\otimes \delta + \delta \otimes \id)\Delta(a).
\end{eqnarray*}
Any such coderivation is called {\it inner}. It follows from \cite[Theorem 3]{Doi} that any coderivation of a coseparable coalgebra is inner. Thus $\delta=0$ is the only coderivation of a coseparable, cocommutative coalgebra, e.g. $C=M_n(k)$.

For an element $a\in R$ of an algebra we denote by $\lambda_a:R\rightarrow R$ the right $R$-linear map given by left multiplication by the element $a$, i.e. $\lambda_a(x)=ax$.

\begin{definition}\label{coderivation_def} Let $R$ be a weak bialgebra and $g, h \in G_w(R)$. A map $\delta:R\rightarrow R$ is called $(g,h)$-coderivation if $\Delta\delta = (\lambda_g\otimes \delta + \delta\otimes \lambda_h)\Delta$. A skew-coderivation is a $(g,h)$-coderivation for some $g,h\in G_w(R)$.
\end{definition}

\begin{lemma} \label{coderivation_epsilon}
Let $\delta$ be a $(g,h)$-coderivation of a weak bialgebra $R$ such that $\epsilon_s(g)=\epsilon_s(h)=1$. Then $\epsilon\delta=0$.
\end{lemma}
\begin{proof} For $a\in R$ we have
$$\delta(a)=(\epsilon\otimes \id)\Delta\delta(a)
=(\epsilon\otimes \id)(ga_1\otimes \delta(a_2) + \delta(a_1)\otimes ha_2)
=\delta(a) + \epsilon(\delta(a_1))ha_2
$$
Hence $\epsilon(\delta(a_1))ha_2=0$. Applying $\epsilon$ again and using equation (\ref{wcounit}) leads to $0 = \epsilon(\delta(a_1))\epsilon(ha_2) = \epsilon(\delta(a_1))\epsilon(a_2)) = \epsilon(\delta(a))$.
\end{proof}

\section{The main results}

Before we present necessary and sufficient conditions to extend a weak Hopf algebra structure from an algebra $R$ to an Ore extension $H=R[x;\sigma,\delta]$ we would like to point out that it is possible to apply Panov's original characterisation for Hopf algebras to produce proper weak Hopf algebra structures on Ore extensions.

Let $A$ be any Hopf algebra with automorphism $\sigma$, $\sigma$-derivation $\delta$ and group-like element $g$. Suppose that the Hopf algebra structure of $A$ can be extended to the Ore extension $A[x;\sigma, \delta]$ such that $x$ is $(g,1)$-primitive, e.g. if this data satisfies Panov's criteria. Then given any weak Hopf algebra $R$, the algebra $R\otimes A$ has a natural structure of weak Hopf algebra with group like element $\overline{g}=1\otimes g$. Moreover,  $\overline{\sigma} = id_R \otimes \sigma$ is an automorphism of $R\otimes A$ and $\overline{\delta}=id_R\otimes \delta$ is an $\overline{\sigma}$-derivation of $R\otimes A$. Since $(R\otimes A)[y;\overline{\sigma},\overline{\delta}]$ is isomorphic to $R\otimes A[x;\sigma, \delta]$ with $y\mapsto 1\otimes x$, the weak Hopf algebra structure of $R\otimes A[x;\sigma,\delta]$ can be lifted to a weak Hopf algebra structure on $(R\otimes A)[y;\overline{\sigma},\overline{\delta}]$ which extends the weak Hopf algebra structure of $R\otimes A$ and would be an ordinary Hopf algebra if and only if $R$ was an ordinary Hopf algebra.

A particular example is the groupoid algebra $H=M_n(kG)$ for some group $G$, field $k$ and $n>0$. Since $H=R\otimes A$ with $R=M_n(k)$ a weak Hopf algebra and $A=kG$ an ordinary Hopf algebra, we have that if $kG[x;\sigma,\delta]$ has a Hopf algebra structure extending the one of $kG$, then $M_n(kG)[y,\overline{\sigma},\overline{\delta]}$ becomes a weak Hopf algebra extending the weak Hopf algebra structure of $M_n(kG)$.

\begin{proposition}[Necessary conditions]\label{nec_cond_0}
Let $R$ be a weak bialgebra with automorphism $\sigma$, 
$\sigma$-derivation $\delta$ and weak group-like element $g$, such that the weak bialgebra structure on $R$ extends to a weak bialgebra structure on $H=R[x;\sigma,\delta]$ with $x$ an $(g,1)$-primitive element. Then
\begin{enumerate}
\item[(i)]  $\epsilon_t(g)=1$;
\item[(ii)] $\delta$ is a $(g,1)$-coderivation; 
\item[(iii)] $\sigma = \tau_\chi^l$, for some weak left character $\chi\in X_w^l(R)$, which has a right inverse in $R^*$;
\item[(iv)] $\Delta(\sigma(a))(g\otimes 1) = (g\otimes 1)(\id\otimes \sigma)\Delta(a)$
holds for any $a\in R$.
\end{enumerate}
\end{proposition}

\begin{proof} 
 By Remark \ref{remark_epsilon_t}, $\epsilon_t(g)=1$. By equation (\ref{wcounit}) we have therefore $\epsilon(Hx)=0$.
For any $a\in R$ we have:
\begin{eqnarray*}\Delta(x)\Delta(a) &=& (g\otimes x + x\otimes 1)\Delta(a) \\
&=& (ga_1\otimes \sigma(a_2))(1\otimes x) +  ga_1\otimes \delta(a_2)  + (\sigma(a_1)\otimes a_2)(x\otimes 1) + \delta(a_1) \otimes a_2\\
&=& (ga_1\otimes \sigma(a_2))(1\otimes x)  + (\sigma(a_1)\otimes a_2)(x\otimes 1) +  ga_1\otimes \delta(a_2)  + \delta(a_1) \otimes a_2
\end{eqnarray*}
On the other hand using the multiplication rule in $H$:
$$\Delta(x)\Delta(a) = \Delta(\sigma(a))\Delta(x) + \Delta(\delta(a)) =  \Delta(\sigma(a))(g\otimes 1)(1\otimes x) +  \Delta(\sigma(a))(x\otimes 1) + \Delta(\delta(a)) $$
Hence comparing the coefficients of $1\otimes x$, $x\otimes 1$ and $1\otimes 1$ we conclude
\begin{eqnarray}\Delta(\sigma(a))(g\otimes 1) &=& ga_1\otimes \sigma(a_2) \label{eq.1} \\
\Delta(\sigma(a)) &=& \sigma(a_1)\otimes a_2 \label{eq.2}\\
\Delta(\delta(a)) &=& ga_1\otimes \delta(a_2)  + \delta(a_1) \otimes a_2.\label{eq.3}
\end{eqnarray}
The first condition is equation (iv), while the last condition shows that $\delta$ is a $(g,1)$-coderivation. Set $\chi=\epsilon\circ\sigma
\in R^{*}$. Applying $\epsilon\otimes \id$ to (\ref{eq.2}) yields $\sigma(a)=\chi(a_1)a_2 = \tau_\chi^l(a)$, i.e. $\sigma=\tau_\chi^l$. Since $\sigma$ is an automorphism, $\chi\in X_w^l(R)$ and $\chi$ has a right inverse, by Lemma \ref{chi_rs_linear_part3}.
\end{proof}

For an invertible element $g\in R$ the linear automorphism of $R$ that sends an element $a$ to $gag^{-1}$  is called the (left) adjoint map by $g$ and denoted by $Ad_g$.

\begin{theorem}[Extending the weak bialgebra structure]\label{nec_cond_1}
Let $R$ be a weak bialgebra with automorphism $\sigma$, 
$\sigma$-derivation $\delta$ and group-like element $g$ with 
$\epsilon(a\delta(b))=0$ for all $a,b\in R$.

Then the following statements are equivalent:
\begin{enumerate}
\item[(a)] The weak bialgebra structure on $R$ extends to a weak bialgebra structure on $H=R[x;\sigma,\delta]$ such that $x$ is a $(g,1)$-primitive element.
\item[(b)]  There exists character $\chi$ such that $\sigma = \tau_\chi^l = Ad_g \tau_\chi ^r$ and $\delta$ is a $(g,1)$-coderivation.
\end{enumerate}
\end{theorem}
\begin{proof}
$(a)\Rightarrow (b)$ By Proposition \ref{nec_cond_0} $\delta$ is a $(g,1)$-coderivation and  there exists a weak left character $\chi=\epsilon\sigma$ such that $\sigma=\tau_\chi^l$. Moreover by the same Proposition, one has $\Delta(\sigma(a))(g\otimes 1) = (g\otimes 1)(\id\otimes \sigma)\Delta(a)$, for any $a\in R$.  Multiplying $g^{-1}\otimes 1$ from the left yields
$\Delta\sigma = (Ad_g\otimes \sigma)\Delta$ and applying $(id\otimes \epsilon)$ to both sides from the right brings us to $\sigma = Ad_g\tau_{\epsilon\sigma}^r$. Since
 $\epsilon\sigma = \epsilon\tau_\chi^l = \chi$, we have $\sigma=Ad_g\tau_\chi^r$ or equivalently $\tau_\chi^r=Ad_{g^{-1}}\sigma$, which shows that $\tau_\chi^r$ is also an automorphism and hence $\chi\in X_w^r(R)$ is also a weak right character and $\chi$ has a right inverse, i.e. $\chi$ is a character.

$(b)\Rightarrow (a)$: 
Suppose $\sigma=Ad_g\tau_\chi^r$, then for any $a\in R$:
$$\Delta(\sigma(a))(g\otimes 1) = \Delta(ga_1\chi(a_2)g^{-1})(g\otimes 1)
= (ga_1g^{-1})\otimes (ga_2g^{-1})(g\otimes 1)\chi(a_3) = (g\otimes 1) (a_1\otimes \sigma(a_2)).$$
Assuming $\epsilon(a\delta(b))=0$ for all $a,b\in R$, we also calculate for all $a\in R$:.
\begin{eqnarray*}
(\dagger)\qquad (g\otimes x + x\otimes 1)\Delta(1)\Delta(a)&=&
ga_1\otimes \sigma(a_2)x + ga_1\otimes \delta(a_2) + \sigma(a_1)x\otimes a_2 + \delta(a_1)\otimes a_2\\
&=& \Delta(\sigma(a))(g\otimes x) +  \Delta(\sigma(a))(x\otimes 1)+ \Delta(\delta(a))
\end{eqnarray*}
where we also used that $\delta$ is a $(g,1)$-coderivation, together with  Proposition \ref{charcterisation_char}.
In particular for $a=1$ we obtain $(g\otimes x + x\otimes 1)\Delta(1)=\Delta(1)(g\otimes x + x\otimes 1).$ Moreover, by the universal property of the Ore extension we obtain  a (non-unital) algebra homomorphism $\Delta:H\rightarrow H\otimes H$ that extends the comultiplication of $R$ and satisfies  $\Delta(x)=\Delta(1)(g\otimes x + x\otimes 1)$ since the equation $(\dagger)$ shows that $\Delta(x)\Delta(a)=\Delta(\sigma(a))\Delta(x)+\Delta(\delta(a))$ holds for all $a\in R$.
In order to prove the coassociativity of $\Delta$, we show  $(\Delta\otimes \id)\Delta(x)=(\id\otimes \Delta)\Delta(x)$ since then by the uniqueness part of the universal property of the Ore extension both maps $(\Delta\otimes \id)\Delta$ and $(\id\otimes \Delta)\Delta$ are equal. As a shorthand we write  $\Delta^2(a)=(\Delta\otimes \id)\Delta(a)$ for any $a\in R$. Note that 
$$\Delta^2(a)(\Delta(1)\otimes 1) = \Delta^2(a)(1\otimes \Delta(1)) =\Delta^2(a).$$
Hence we calculate:
\begin{eqnarray*}
(\Delta\otimes \id)\Delta(x)
&=&\Delta^2(1)(\Delta(g)\otimes x + \Delta(x)\otimes 1)\\
&=&\Delta^2(1)(g\otimes g\otimes x + g\otimes x\otimes 1 + x\otimes 1\otimes 1)\\
&=&\Delta^2(1)(g\otimes \Delta(x)\otimes x + x\otimes \Delta(1)) \\
&=&(\id\otimes \Delta)\Delta(x).
\end{eqnarray*}
This shows the coassociativity of $\Delta$.

For the counity, we define a linear map $\epsilon:H\rightarrow k$ by $\epsilon\left(\sum_{i=0}^n a_ix^i\right) = \epsilon(a_0)$ for any $a_i\in R$.
The map $\epsilon$ is well defined since the powers of $x$ form a basis of $H$ over $R$. Note that for all $a,b\in R$:
$\epsilon(axb)=\epsilon(a\sigma(b)x)+\epsilon(a\delta(b))=0$, where we use our additional assumption that $\epsilon(a\delta(b))=0$ holds for any $a,b\in R$.
By induction it is easy to prove $\epsilon(HxH)=0$.

It is enough to verify the property of the counit for elements of the form $ax^n$ with $n>0$.
A short induction argument shows that for any $n>0$, there exist elements $C^{(n)}_{i,j} \in R$ such that 
\begin{equation}\label{eq_expand_power}
(g\otimes x + x\otimes 1)^n = \sum_{i,j=0}^n C^{(n)}_{i,j} x^i\otimes x^j,
\end{equation}
where 
 $C^{(n)}_{n,0}=1$ and $C^{(n)}_{i,0}=0$ for any $i<n$, as well as 
$ C^{(n)}_{0,n}=g^n$ and $C^{(n)}_{0,j} \in \mathrm{span}\{a\delta(b):a,b\in R\}$ for any $j<n$.
Then
$$(\id\otimes \epsilon)\Delta(ax^n) =\sum_{i,j=0}^n a_1C^{(n)}_{i,j} x^i \epsilon(a_2x^j) 
=\sum_{i=0}^n aC^{(n)}_{i,0} x^i  = ax^n,$$
since $C^{(n)}_{i,0}=0$ for $i<n$ and $C^{(n)}_{n,0}=1$. Moreover, 

$$(\epsilon\otimes \id)\Delta(ax^n) =\sum_{i,j=0}^n \epsilon(a_1C^{(n)}_{i,j} x^i)a_2x^j = 
\underbrace{\epsilon(a_1g^n)}_{=\epsilon(a_1)}a_2x^n + \sum_{j=1}^n \underbrace{\epsilon(a_1C^{(n)}_{0,j})}_{=0}a_2x^j = ax^n.$$
where we use Lemma \ref{property_et1} for $\epsilon(ag^n)=\epsilon(a)$, since $\epsilon_t(g)=1$ by $(iv)$ and $\epsilon(a_1C^{(n)}_{0,j})=0$ as by equation (\ref{eq_expand_power}), $a_1C^{(n)}_{0,j}\in \mathrm{span}\{a\delta(b):a,b\in R\}$ and $\epsilon(a\delta(b))=0$, by our additional assumption for $(b)\Rightarrow (a)$. 
Hence $H$ is coalgebra.

Now, we are going to show the weak counity property. Let $\alpha=\tilde{\alpha}x+a$, $\beta=\tilde{\beta}x+b$ and $\gamma=\tilde{\gamma}x+c \in H$. 
Using $\epsilon(HxH)=0$, what we have mentioned earlier, it is easy to check that $\epsilon(\alpha\beta\gamma)=\epsilon(abc)$. Note that
$\Delta(\beta)=\Delta(\tilde{\beta})(g\otimes x + x\otimes 1) + \Delta(b).$ 
Thus 
$$\epsilon(\alpha\beta_1)\epsilon(\beta_2\gamma) = \epsilon(\alpha\tilde{\beta}_1g)\underbrace{\epsilon(\tilde{\beta}_2x\gamma)}_{=0} + 
\underbrace{\epsilon(\alpha\tilde{\beta}_1x)}_{=0}\epsilon(\tilde{\beta}_2\gamma) + \epsilon(\alpha b_1)\epsilon(b_2\gamma) 
= \epsilon(\alpha b_1)\epsilon(b_2\gamma)  = \epsilon(ab_1)\epsilon(b_2c)
$$
Similarly $\epsilon(\alpha\beta_2)\epsilon(\beta_1\gamma)  = \epsilon(\alpha b_2)\epsilon(b_1\gamma)= \epsilon(ab_2)\epsilon(b_1c)$. Since $R$ is a weak bialgebra we have
$$\epsilon(\alpha\beta_1)\epsilon(\beta_2\gamma) = \epsilon(ab_1)\epsilon(b_2c)
= \epsilon(abc)=\epsilon(ab_2)\epsilon(b_1c)=\epsilon(\alpha\beta_2)\epsilon(\beta_1\gamma).$$
This shows that $H$ is a weak bialgebra extending the structure of $R$.

 \end{proof}

\begin{remark}\label{remark_epsilon_delta}
The extra condition $\epsilon(a\delta(b))=0$, for all $a,b\in R$ in the Theorem, seems to be strange and it is not clear to the authors whether it is really needed. 
If $g$  a group-like element and hence invertible, then $\epsilon_s(g)=1$ by Lemma \ref{basics_grplke}. Moreover, if $\delta$ is a $(g,1)$-coderivation and $\epsilon_s(g)=1$, then  
$\epsilon\delta=0$ by  Lemma \ref{coderivation_epsilon}. If we had $\delta(R_s)=0$ and $\sigma=\tau_\chi^l$ for some weak left character $\chi\in X_w^l(R)$, then for all $a,b\in R$:
$$\epsilon(a\delta(b)) = \epsilon(\epsilon_s(a)\delta(b)) = \epsilon(\sigma(\epsilon_s(a))\delta(b)) = \epsilon(\delta(\epsilon_s(a)b)) - \epsilon(\delta(\epsilon_s(a))b)= 0,$$
where we use equation (\ref{wcounit}) in the first equality, the $R_s$-linearity of $\tau^l_\chi$ in the second equality (see Lemma  \ref{chi_rs_linear_part3} ) and  $\delta(R_s)=0$ and $\epsilon\delta = 0$ in the last one.
Clearly $\delta(R_s)=0$ for an arbitrary Hopf algebra, or for $\delta=0$, i.e. in case  $R$ is a coseparable, cocommutative weak Hopf algebra like a groupoid algebra.\end{remark}

\begin{theorem}\label{main_result} Let $R$ be a weak Hopf algebra with antipode $S$, $\sigma$ an automorphism of $R$, $\delta$ an $\sigma$-derivation of $R$ and $g$ a group-like element.
	Suppose that $\delta(R_s)=0$. Then the following conditions are equivalent:
\begin{enumerate}
\item[(a)] The weak Hopf algebra structure of $R$ extends to a weak Hopf algebra structure on $H=R[x;\sigma,\delta]$, such that $x$ is $(g,1)$-primitive element and $S(x)=-S(g)x$.
\item[(b)] The following statements are satisfied:
\begin{enumerate}
	\item[(i)] $\sigma=\tau_\chi^l = Ad_g\tau_\chi^r$, for some character $\chi$ of $R$;
	\item[(ii)] $\delta$ is a $(g,1)$-coderivation;	
	\item[(iii)] $Ad_g S = \sigma S \sigma$;	
	\item[(iv)] $\delta S \sigma = \lambda_g S \delta$.
\end{enumerate}
\end{enumerate}
\end{theorem}

\begin{proof} $(a)\Rightarrow (b)$ 
	By Proposition \ref{nec_cond_0}, $\sigma=\tau_\chi^l$ for some weak left characters $\chi$ and $\delta$ is a $(g,1)$-coderivation. Hence by Remark \ref{remark_epsilon_delta}, $\epsilon(a\delta(b))=0$ for all $a,b\in R$. Thus, we can apply Theorem \ref{nec_cond_1} to obtain that $\chi$ is a character. 
	Under the assumption that $S(x)=-S(g)x$ and using the commutation rule in $H=R[x;\sigma,\delta]$ we calculate for any $a\in R$:
\begin{eqnarray*}
S(a)S(g)x &=& -S(xa) \\
&=& -S(\sigma(a)x) - S(\delta(a)) \\
&=& S(g)xS(\sigma(a)) - S(\delta(a)) 
= S(g)\sigma(S(\sigma(a)))x +[S(g)\delta(S(\sigma(a)))-S(\delta(a))]
\end{eqnarray*}
Comparing the coefficients of $x$ and $1$ leads to the equations
$S(a)S(g) = S(g)\sigma(S(\sigma(a)))$ and $S(g)\delta(S(\sigma(a)))=S(\delta(a))$ for all $a\in R$. Hence $Ad_gS = \sigma S \sigma$, since $S(g)=g^{-1}$  by Lemma \ref{basics_grplke}, and $\delta S\sigma = \lambda_g S\delta$.

$(b)\Rightarrow (a)$: Let $\sigma=\tau_\chi^l$ for some character $\chi$. Using the assumptions $\delta(R_s)=0$ and $\epsilon \circ \delta = 0$, we have  $\epsilon(a\delta(b))=0$ for all $a,b\in R$, 
by  Remark \ref{remark_epsilon_delta}. Hence by Proposition \ref{nec_cond_1} the weak bialgebra structure of $R$ extends to $H=R[x;\sigma,\delta]$ with 
$x$ being $(g,1)$-primitive and $\epsilon(Hx)=0$. Since $\epsilon(axb)=\epsilon(a\sigma(b)x)+\epsilon(a\delta(b))=0$ for all $a,b\in R$ we also have $\epsilon(HxH)=0$.

Let $\gamma=-S(g)x$. Using properties $(iii)$ and $(iv)$ we calculate for all $a\in R$:
\begin{eqnarray*}
 S(a)\gamma 
 &=& -S(a)S(g)x - S(\delta(a))+ S(\delta(a)) \\
 &=& -S(g)\sigma(S(\sigma(a)))x - S(g)\delta(S(\sigma(a))) + S(\delta(a))\\
 &=& -S(g)(xS(\sigma(a))) + S(\delta(a))\\
 &=& \gamma S(\sigma(a)) + S(\delta(a)) 
\end{eqnarray*}
Hence by the universal property of the Ore extension, there exists a unique unital anti-algebra homomorphism $S:H\rightarrow H$ that extends $S$ and sends $x$ to $S(x)=\gamma=-S(g)x.$ We will verify the antipode axioms of a weak Hopf algebra.
It is clear that $a_1S(a_2)=\epsilon_t(a)$, $S(a_1)a_2=\epsilon_s(a)$ and $S(a_1)a_2S(a_3)=S(a)$ holds for all $a\in R$.
For any $h\in H$ we have
$\epsilon_t(hx)=\epsilon(1_1hx)1_2=0$, since $\epsilon(Hx)=0$. Therefore
$$(hx)_1S((hx)_2) = h_1gS(h_2x)+h_1xS(h_2) = -h_1gS(g)xS(h_2)+h_1xS(h_2)=0=\epsilon_t(hx),$$
as $gS(g)=\epsilon_t(g)=1$ by assumption $(2.i)$. 

Similar, $\epsilon_s(hx)=\epsilon(hx1_1)1_2=0$, since $\epsilon(HxH)=0$.
Furthermore note that the elements of $R_s$ commute with $x$, because $\sigma=\tau^l_\chi$ is $R_s$-linear by Lemma \ref{chi_rs_linear_part3}  and $\delta(R_s)=0$ by assumption, i.e.
$xa=\sigma(a)x+\delta(a)=ax$ for any $a\in R_s$.
Hence using $S(x)=-S(g)x$ and $S(a_1)a_2\in R_s$ we calculate for any $a\in R$:
$S((ax)_1)(ax)_2 = S(a_1g)a_2x + S(a_1x)a_2= S(g)(\epsilon_s(a)x-x\epsilon_s(a)) = 0.$
Suppose that $S(h_1)h_2=0$ for all $h=ax^n$ with $a\in R$ and some $n\geq 0$. Then we also calculate:
$S((hx)_1)(hx)_2 = S(h_1g)h_2x + S(h_1x)h_2= S(g)S(h_1)h_2x+S(x)S(h_1)h_2=0.$
Hence by induction on $n$ we have that $S(h_1)h_2=\epsilon_s(h)$ for any $h=ax^n$ with $a\in R$ and $n\geq 0$.
In particular we have that $H_s=\epsilon_s(H) = \epsilon_s(R)=R_s$, i.e. all elements of $H_s$ commute with $x$ in $H$.

Finally we will show 
\begin{equation}\label{third_antipode} S(h_1)h_2S(h_3)=S(h)\end{equation} for all $h\in H$. Again it is enough to show this only for the monomials $h=ax^n$ with $n\geq 0$.
For $n=0$ and $h=a\in R$ we have $S(h_1)h_2S(h_3)=S(h)$ since $S$ extends the antipode of $R$. 
Suppose that (\ref{third_antipode}) holds for all elements $h=ax^n$ for some $n\geq 0$. Then
\begin{eqnarray*}
S((hx)_1)(hx)_2S((hx)_3) &=& S(h_1g)h_2gS(h_3x) + S(h_1g)h_2xS(h_3) + S(h_1x)h_2S(h_3)\\
 &=& S(g)S(h_1)h_2gS(x)S(h_3) + S(g)S(h_1)h_2xS(h_3) + S(x)S(h_1)h_2S(h_3) \\
  &=& -S(g)\epsilon_s(h_1)gS(g)xS(h_2) + S(g)\epsilon_s(h_1)xS(h_2) + S(x)\epsilon_s(h_1)S(h_2) \\
    &=& -S(g)x\epsilon_s(h_1)S(h_2) + S(g)x\epsilon_s(h_1)S(h_2) + S(x)S(h) \\
    &=& S(x)S(h) =S(hx)
\end{eqnarray*}

\end{proof}

\begin{remark}
Condition $(b.iii)$ is automatically satisfied in case $g$ is invertible and the inverse of $\chi$ is of the form $\chi S$, which is always true for an ordinary Hopf algebra. This is so, because in this case, by Corollary \ref{chi_rs_linear}, one has that $S=\tau_\chi^r   S \tau_\chi^l$. Since by condition $(b.i$), $\sigma=\tau_\chi^l=Ad_g\tau_\chi^r$, we have
$Ad_g S = \sigma S \sigma$, i.e. $gS(a)S(g) = \sigma(S(\sigma(a)))$, for all $a\in R$. 
\end{remark}

\section{Ore extensions of connected groupoid algebras}

We finish by discussing the Ore extensions of the groupoid algebra $k\cG$ of a ``connected'' groupoid $\cG$, by which we mean a groupoid such that any two objects are connected by a morphism. We furthermore assume the set of objects is finite. Then $R=k\cG=M_n(kG)$, for some $n>0$ and some group $G$.
Panov characterised such Ore extensions for the case $n=1$ (see \cite[Proposition 2.2]{Panov}).

\begin{proposition} Let $R$ be a cocommutative weak Hopf algebra, with automorphism $\sigma$ and $\sigma$-derivation $\delta$ such that the weak Hopf algebra structure of $R$ extends to one of $H=R[x;\sigma,\delta]$  with $x$ being an $(g,1)$-primitive element, for some group-like element $g$, such that $S(x)=-S(g)x$, $\delta(R_s)=0$ and the inverse of $\chi$ is $\chi S$, then $g$ is central.
\end{proposition}

\begin{proof}
Since $\sigma(a)=\tau_\chi^l(a)=Ad_g\tau_\chi^r(a)$ by Theorem \ref{main_result}, we have
$$a=\chi(S(a_1))\sigma(a_2) = Ad_g\sigma(\sigma^{-1}(a)) = Ad_g(a).$$
Hence $ag=ga$, for all $a\in R$. \end{proof}

The Proposition generalises \cite{Panov}*{Corollary 1.4}. Moreover, the last Proposition applies in particular to groupoid algebras $k\cG$ of a connected groupoid with finitely many objects, i.e. $R=M_n(kG)$ for some $n>0$ and group $G$. Generalising Example \ref{example_matrix}, any weak group-like element is of the form $
\gamma=\sum_{i\in I} g_i E_{i,\sigma(i)}$ for some subset $I\subseteq \{1,\ldots, n\}$ and injective map $\sigma:I\rightarrow \{1,\ldots, n\}$ and elements $g_i\in G$. Moreover, if $\gamma$ has a left inverse, then $I=\{1,\ldots, n\}$ and $\sigma$ has to be a bijection. In particular $\gamma$ has also a right inverse which is 
$S(\gamma)=\sum_{i=1}^n g_i^{-1} E_{\sigma(i),i}.$ Furthermore the  characters $\chi\in X_w(R)$ of $R=M_n(kG)$ are of the form $\chi(gE_{ij})=p_i^{-1}p_j \rho(g)$, where $\rho\in kG^*$ is a group character and $p_1,\ldots, p_n\in k^{\times}$ are non-zero scalars. Clearly $\chi S$ is the inverse of $\chi$.

Let $R$ be a cocommutative weak Hopf algebra, $\chi \in X(R)$ and $g$ a central group-like element of $R$. Suppose that  $\alpha\in R^*$  satisfies 
\begin{equation} \alpha(ab)=\alpha(a)\epsilon(b) + \chi(a)\alpha(b), \qquad \forall a,b\in R. \end{equation}
Then $\tau_\alpha^l$ satisfies
\begin{equation} \tau_\alpha^l(ab)=\alpha(a_1b_1)a_2b_2 = \alpha(a_1)a_2\epsilon(b_1)b_2 + \chi(a_1)a_2\alpha(b_1)b_2 = \tau_\alpha^l(a) b + \tau_\chi^l(a)\tau_\alpha^l(b).\end{equation}
Define the map $\delta:R\rightarrow R$, by $\delta(a):=(1-g)\tau_{\alpha}^l(a)$, for all $a\in R$.
Since $1-g$ is central, we have $\delta(ab)=\delta(a)b + \tau_\chi^l(a)\delta(b)$, i.e. $\delta$ is an $\tau_\chi^l$-derivation.
Furthermore $\delta$ is a $(g,1)$-coderivation, because 
\begin{eqnarray*}
\Delta\delta(a) &=& \alpha(a_1)\Delta(1-g)\Delta(a_2)\\
& =&  \alpha(a_1)( (1-g)\otimes 1 + g\otimes(1-g))\Delta(a_2)\\
& = &  \alpha(a_1)(1 - g)a_2 \otimes a_3 +  g a_1 \otimes (1 - g) \alpha(a_2)a_3\\
&=&  \delta(a_1)\otimes a_2 + ga_1 \otimes \delta(a_2) = (\delta\otimes \id + \lambda_g\otimes \delta)\Delta(a)
\end{eqnarray*}  
where we use the cocommutativity in the third equation. Note that $Ad_g=\id$ as $g$ is central and $\tau_\chi^l=Ad_g\tau_\chi^r = \tau_\chi^r$, as $R$ is cocommutative.

\begin{theorem} Let $G$ be any group, $n>0$, $R=M_n(kG)$, $\chi \in X(R)$ and $g$ a central element of $G$. 
Suppose that there exists $\alpha \in R^*$ such that $\alpha(ab)=\alpha(a)\epsilon(b) + \chi(a)\alpha(b)$ holds for any $a,b\in R$ and $\alpha(E_{ii})=0$ for any $i$. Let $\sigma=\tau_\chi^l$ and $\delta(a):= (1-g)\tau_\alpha^l(a)$, for all $a\in R$.
Then  $\delta$ is a $\sigma$-derivation and a $(g,1)$-coderivation such that
$R[x;\tau_\chi^l,\delta]$ becomes a weak Hopf algebra with 
$$\Delta(x)=\Delta(1)(g\otimes x + x\otimes 1), \qquad \epsilon(HxH)=0,\qquad S(x)=-g^{-1}x.$$
\end{theorem}

\begin{proof}

We first observe that $\delta(R_s) = 0$. In fact, since every element of $R_s$ is of the form $\sum_i \lambda_iE_{ii}$, with $\lambda_i \in k$, we have
$\delta\left(\sum_i \lambda_iE_{ii}\right) = \sum_i \lambda_i\delta(E_{ii}) = \sum_i \lambda_i (1-g)\alpha(E_{ii})E_{ii} = 0.$
Hence, in order to obtain our result, it is enough to check the conditions of the implication $(b)\Rightarrow (a)$ of Theorem \ref{main_result}. Conditions $(i)$ and $(ii)$ hold by precedent discussion and because $g$ is an invertible central element. Also, since $\chi S$ is the inverse of $\chi$, the condition $(iii)$ follows by Corollary \ref{chi_rs_linear}. Thus, the only work to do is to prove the condition $(iv)$.

We begin by noting that under our hypotheses we have $\alpha(g_{ij}E_{ij}) = -\chi(g_{ij}E_{ij})\alpha(g_{ij}^{-1}E_{ji})$. In fact, because 
\begin{eqnarray*}
0 &=& \alpha(E_{ii}) = \alpha(g_{ij}E_{ij}g_{ij}^{-1}E_{ji}) = \alpha(g_{ij}E_{ij})\epsilon(g_{ij}^{-1}E_{ji}) + \chi(g_{ij}E_{ij})\alpha(g_{ij}^{-1}E_{ji}) \\
&=& \alpha(g_{ij}E_{ij}) + \chi(g_{ij}E_{ij})\alpha(g_{ij}^{-1}E_{ji}).
\end{eqnarray*}

Let $a = \sum_{i,j} g_{ij}E_{ij} \in R$. Then we obtain
\begin{eqnarray*}
S(g)\delta(S(\sigma(a))) &=& g^{-1}\sum_{i,j=1}^{n}\chi(g_{ij}E_{ij})\delta(S(g_{ij}E_{ij})) = g^{-1}\sum_{i,j=1}^{n}\chi(g_{ij}E_{ij})\delta(g_{ij}^{-1}E_{ji}) \\
&=&  g^{-1}\sum_{i,j=1}^{n}\chi(g_{ij}E_{ij})\delta(g_{ij}^{-1}E_{ji}) = g^{-1}\sum_{i,j=1}^{n}\chi(g_{ij}E_{ij})(1-g)\alpha(g_{ij}^{-1}E_{ji})g_{ij}^{-1}E_{ji} \\
&=& (g^{-1}-1)\sum_{i,j=1}^{n}\chi(g_{ij}E_{ij})\alpha(g_{ij}^{-1}E_{ji})g_{ij}^{-1}E_{ji} = (1-g^{-1})\sum_{i,j=1}^{n}\alpha(g_{ij}E_{ij})g_{ij}^{-1}E_{ji} \\
&=& \sum_{i,j=1}^{n}\alpha(g_{ij}E_{ij})(g_{ij}^{-1}E_{ji} - g^{-1}g_{ij}^{-1}E_{ji}) = \sum_{i,j=1}^{n}\alpha(g_{ij}E_{ij})S((1-g)g_{ij}E_{ij}) \\
&=& \sum_{i,j=1}^{n} S((1-g)\alpha(g_{ij}E_{ij})g_{ij}E_{ij}) = S(\delta(\sum_{i,j=1}^{n} g_{ij}E_{ij})) \\ &=& S(\delta(a)).
\end{eqnarray*}
Multiplying by $g$ from the left yields $\delta(S(\sigma(a))) = gS(\delta(a))$ and the proof is complete.

\end{proof}


\section*{References}

 \begin{biblist}
 \bib{Boehm}{article}{
   author={B\"ohm, Gabriella},
   author={Nill, Florian},
   author={Szlach\'anyi, Korn\'el},
   title={Weak Hopf algebras. I. Integral theory and $C^*$-structure},
   journal={J. Algebra},
   volume={221},
   date={1999},
   number={2},
   pages={385--438},
   issn={0021-8693},
   review={\MR{1726707}},
   doi={10.1006/jabr.1999.7984},
}

\bib{BrownGoodearl}{book}{
   author={Brown, Ken A.},
   author={Goodearl, Ken R.},
   title={Lectures on algebraic quantum groups},
   series={Advanced Courses in Mathematics. CRM Barcelona},
   publisher={Birkh\"auser Verlag, Basel},
   date={2002},
   pages={x+348},
   isbn={3-7643-6714-8},
   review={\MR{1898492}},
   doi={10.1007/978-3-0348-8205-7},
}
	
		
\bib{Brown}{article}{
   author={Brown, K. A.},
   author={O'Hagan, S.},
   author={Zhang, J. J.},
   author={Zhuang, G.},
   title={Connected Hopf algebras and iterated Ore extensions},
   journal={J. Pure Appl. Algebra},
   volume={219},
   date={2015},
   number={6},
   pages={2405--2433},
   issn={0022-4049},
   review={\MR{3299738}},
   doi={10.1016/j.jpaa.2014.09.007},
}

	

\bib{BrzezinskiWisbauer}{book}{
   author={Brzezinski, Tomasz},
   author={Wisbauer, Robert},
   title={Corings and comodules},
   series={London Mathematical Society Lecture Note Series},
   volume={309},
   publisher={Cambridge University Press, Cambridge},
   date={2003},
   pages={xii+476},
   isbn={0-521-53931-5},
   review={\MR{2012570}},
   doi={10.1017/CBO9780511546495},
}

\bib{Doi}{article}{
   author={Doi, Yukio},
   title={Homological coalgebra},
   journal={J. Math. Soc. Japan},
   volume={33},
   date={1981},
   number={1},
   pages={31--50},
   issn={0025-5645},
   review={\MR{597479}},
   doi={10.2969/jmsj/03310031},
}
	

\bib{Nikshych}{article}{
   author={Nikshych, Dmitri},
   title={On the structure of weak Hopf algebras},
   journal={Adv. Math.},
   volume={170},
   date={2002},
   number={2},
   pages={257--286},
   issn={0001-8708},
   review={\MR{1932332}},
   doi={10.1006/aima.2002.2081},
}

\bib{NikshychVainerman}{article}{
   author={Nikshych, Dmitri},
   author={Vainerman, Leonid},
   title={Finite quantum groupoids and their applications},
   conference={
      title={New directions in Hopf algebras},
   },
   book={
      series={Math. Sci. Res. Inst. Publ.},
      volume={43},
      publisher={Cambridge Univ. Press, Cambridge},
   },
   date={2002},
   pages={211--262},
   review={\MR{1913440}},
}
	

\bib{Panov}{article}{
   author={Panov, A. N.},
   title={Ore extensions of Hopf algebras},
   language={Russian, with Russian summary},
   journal={Mat. Zametki},
   volume={74},
   date={2003},
   number={3},
   pages={425--434},
   issn={0025-567X},
   translation={
      journal={Math. Notes},
      volume={74},
      date={2003},
      number={3-4},
      pages={401--410},
      issn={0001-4346},
   },
   review={\MR{2022506}},
   doi={10.1023/A:1026115004357},
}
	
\end{biblist}
\end{document}